\documentclass[twoside,12pt, letterpaper,reqno]{amsart}%
\usepackage[dvips]{graphics}
\usepackage{times}
\usepackage{mathrsfs}
\usepackage[T1]{fontenc}
\usepackage{latexsym}
\usepackage{epsfig}
\usepackage{hyperref}
\hypersetup{
    pdfborder = {0 0 0}
}
\usepackage{amsmath,amsfonts,amsthm,amssymb,amscd}
\usepackage{pstricks}
\usepackage[myheadings]{fullpage}
\usepackage{mathtools}
\usepackage{braket}
\usepackage{enumerate}
\usepackage{scrextend}
\usepackage{hyperref}
\usepackage{tikz-cd}
\usepackage{subcaption}
\usepackage{yhmath}
\usepackage{float}
\usepackage[english]{babel}
\usetikzlibrary{babel, decorations.pathreplacing}

\usepackage{csquotes}
\MakeOuterQuote{"}
\usepackage{todonotes}

\usepackage{wrapfig}
\usepackage{forest}

\newcommand\C{\mathbb{C}}

\newcommand\Q{\mathbb{Q}}

\DeclarePairedDelimiter\abs{\lvert}{\rvert}%

\DeclareMathOperator{\Spec}{Spec}
\DeclareMathOperator{\Min}{Min}

\makeatletter
\newcommand{\vast}{\bBigg@{4}}
\newcommand{\Vast}{\bBigg@{5}}
\makeatother

\newtheorem{theorem}{Theorem}[section]

\newtheorem{lemma}[theorem]{Lemma}

\newtheorem{proposition}[theorem]{Proposition}

\theoremstyle{definition}

\newtheorem{example}[theorem]{Example}
\newtheorem{remark}[theorem]{Remark}

\newtheorem{question}{Question}

\hypersetup{
	colorlinks = true, %Colours links instead of ugly boxes
	urlcolor = blue, %Colour for external hyperlinks
	linkcolor = blue, %Colour of internal links
	citecolor = red %Colour of citations
}
\DeclareSymbolFont{yhlargesymbols}{OMX}{yhex}{m}{n}\DeclareMathAccent{\yhwidehat}{\mathord}{yhlargesymbols}{"62}

\definecolor{pastelgreen}{HTML}{a8e6cf}
\definecolor{pastelorange}{HTML}{ffd3b6}
\definecolor{pastelpink}{HTML}{ffaaa5}

\tikzstyle{main} = [circle, fill, inner sep=1.5pt]

\begin{document}

\title{Constructing Noncatenary Quasi-Excellent Precompletions}

\author[Ehrenworth]{Jackson Ehrenworth}
\address{Department of Mathematics and Statistics, Williams College, Williamstown, MA 01267}
\email{jacksonehrenworth@gmail.com}

\author[Loepp]{S. Loepp}
\address{Department of Mathematics and Statistics, Williams College, Williamstown, MA 01267}
\email{sloepp@williams.edu}

%\maketitle

\begin{abstract}
Let $T$ be a local (Noetherian) ring and let $Q_1$ and $Q_2$ be prime ideals of $T$. We find sufficient conditions for there to exist a quasi-excellent local subring $B$ of $T$ satisfying the following conditions: (1) the completion of $B$ at its maximal ideal is isomorphic to the completion of $T$ at its maximal ideal, (2) $B \cap Q_1 = B \cap Q_2$, (3) the set of prime ideals of $T/(Q_1 \cap Q_2)$ of positive height is the same as the set of prime ideals of $B/(B \cap Q_1)$ of positive height when viewed as partially ordered sets, and (4) for $i = 1$ and for $i = 2$, there is a coheight preserving bijection between the minimal prime ideals of $T_{Q_i}$ and the minimal prime ideals of $B_{B \cap Q_1}$. Intuitively, this means that $T$ contains a quasi-excellent local subring in which $Q_1$ and $Q_2$ are "glued together" and such that both the completion and desirable properties of the prime spectrum are preserved. 

We use this result to show that certain complete local rings are the completion of a quasi-excellent local ring whose prime spectrum, when viewed as a partially ordered set, contains interesting noncatenary finite subsets.
\end{abstract}

\maketitle

\section{Introduction}

Although completions of Noetherian rings are an important tool in commutative algebra, there are aspects of the relationship between a Noetherian ring and a completion of the ring that remain mysterious. For example, let $R$ be a local (Noetherian) ring and let $\widehat{R}$ be the completion of $R$ with respect to its maximal ideal. The relationship between the prime ideal structure of $R$ and the prime ideal structure of $\widehat{R}$ is not well understood. More specifically, the following question is open.

\begin{question}
\label{question: hard classication question}
Let $T$ be a complete local ring and let $X$ be a partially ordered set (poset). Under what conditions is there a local ring $R$ whose completion is $T$ such that the prime spectrum of $R$, when viewed as a partially ordered set with respect to inclusion, is isomorphic to $X$?
\end{question}

Answering Question \ref{question: hard classication question} would mean that, in some sense, we understand all possible relationships between the prime spectrum of a local ring and the prime spectrum of its completion. We believe that unfortunately, answering Question \ref{question: hard classication question} is very difficult. So, instead of asking what the possibilities there are for the entire prime spectrum of $R$, we ask what the possibilities are for \textit{finite parts} of the prime spectrum of $R$. In other words, we consider the following open question:

\begin{question} 
\label{question: easier classication question}
Let $T$ be a complete local ring and let $X$ be a {\it finite} partially ordered set. Under what conditions is there a local ring $R$ with completion $T$ satisfying the following condition: If $Y$ is the prime spectrum of $R$ when viewed as a partially ordered set with respect to inclusion, there is an injective poset map from $X$ to $Y$ that preserves saturated chains.
\end{question}

Question \ref{question: easier classication question} is particularly interesting when $X$ is not catenary (i.e. when $X$ contains two elements $x_1$ and $x_2$ satisfying the condition that $x_1 < x_2$ and there are two saturated chains of elements starting at $x_1$ and ending at $x_2$ of different lengths), as this would imply that the ring $R$ is not catenary. For this reason, we focus on the case where $X$ is not catenary. While we believe that Question \ref{question: easier classication question} is  difficult to answer in this case, some progress has been made.

In particular, in \cite{Gluing2021} the authors show that given a complete local ring $T$ satisfying certain conditions and given two minimal prime ideals $P_1$ and $P_2$ of $T$, there exists a local subring $R$ of $T$ such that the completion of $R$ at its maximal ideal is $T$, $R \cap P_1 = R \cap P_2$, and the entirety of the rest of the prime spectrum of $R$ is poset isomorphic to the rest of the prime spectrum of $T$ (see Theorem \ref{theorem: GluingTheorem}).  Intuitively, this result says that there exists a local subring $R$ of $T$ whose completion is $T$ and such that, as posets, the prime spectra of $R$ and $T$ are exactly the same with the exception that the two minimal prime ideals $P_1$ and $P_2$ of $T$ are ``glued together'' in $R$. This result can be used to answer Question \ref{question: easier classication question} for particular $T$ and $X$. For example, let $X$ be the poset $\{x_1,x_2,x_3,x_4, x_5\}$ with the relations $x_1 < x_2 < x_3 < x_4$ and $x_1 < x_5 < x_4$, and let $T = \mathbb{C}[[y_1,y_2,y_3,y_4]]/((y_1) \cap (y_2,y_3)).$ Then Theorem \ref{theorem: GluingTheorem} applies and provides a local subring $R$ of $T$ satisfing the conditions given in Question \ref{question: easier classication question} by ``gluing'' the minimal prime ideals of $T$ while maintaing the rest of the prime ideal structure. In fact, the $R$ obtained is quasi-excellent (See Proposition \ref{quasiexcellent}). $R$ cannot be excellent since, by definition, excellent rings are catenary, but $R$ is not catenary.

Theorem \ref{theorem: GluingTheorem} only works, however, when one wants to glue {\it minimal} prime ideals of $T$. In this paper, we consider whether it is possible to glue two prime ideals of $T$ of positive height such that both the completion and certain parts of the prime ideal structure are preserved.  We demonstrate via Theorem \ref{theorem: generalized gluing alt} that in certain cases one can build on the results given in \cite{Gluing2021} to achieve this. Theorem \ref{theorem: generalized gluing alt} can then be used, multiple times in fact, to answer Question \ref{question: easier classication question} for certain $T$ and $X$ where we previously did not know the answer (see Example \ref{mainexample} and Example \ref{secondexample}). Moreover, as is the case in \cite{Gluing2021}, the rings we produce in this setting are quasi-excellent. 

    The strategy, broadly, for proving our main theorem (Theorem \ref{theorem: generalized gluing alt}) is to start with a local ring $T$ satisfying certain very specific nice properties and two prime ideals $Q_1$ and $Q_2$ of $T$, also satisfying some properties. We demonstrate that we can glue $Q_1$ and $Q_2$ together to obtain a local subring $B$ of $T$ such that the completion of $B$ at its maximal ideal is isomorphic to the completion of $T$ at its maximal ideal and such that the prime spectrum of $B$ above $B \cap Q_1$ is the same as the prime spectrum of $T$ above $Q_1 \cap Q_2$. Moreover, the important parts of the prime spectrum that exist in $B$ below $B \cap Q_1$ are the same as the important parts of the prime spectrum of $T$ below both $Q_1$ and $Q_2$.  The ``nice properties'' that $T$ satisfies relate to the ability to ``return'' to $T$ from the quotient ring $T / (Q_1 \cap Q_2)$ by adjoining indeterminates and modding out by an ideal generated by $\Q$-linear combinations of monomials in those indeterminates.  It might seem slightly arbitrary to use $\Q$, but we will need all of our rings to contain a copy of the rationals to use the results of \cite{Gluing2021}.  The ring $T / (Q_1 \cap Q_2)$ has two minimal prime ideals.  We deploy the results of \cite{Gluing2021} to glue them together to obtain a local subring $R$ of $T / (Q_1 \cap Q_2)$, and then use the ability to ``return'' to $T$ from $T/(Q_1 \cap Q_2)$ to get from $R$ to a subring of $T$ that has the same completion as $T$ itself and has our desired properties.

In Section \ref{Prelims}, we present preliminaries, important theorems from the literature that we use for our proofs, and an illustrative example. Our main result is found in Section \ref{MainResult}.

\section{Preliminaries}\label{Prelims}

We begin this section by introducing notation and conventions we use for the remainder of the paper.  When we say a ring is local, we mean that it is Noetherian and it has exactly one maximal ideal.  The notation $(B,M)$ is used for a local ring $B$ whose maximal ideal is $M$. If $(B,M)$ is a local ring, then $\widehat{B}$ is the notation used for the completion of $B$ with respect to $M$. If $(B,M)$ is a local ring and $\widehat{B} \cong T$, we say that $B$ is a precompletion of $T$. We denote the set of prime ideals of a ring $B$ by $\Spec(B)$ and the set of minimal prime ideals of $B$ by $\Min(B)$. If $B$ is a ring and $P$ is a prime ideal of $B$, then the coheight of $P$ in $B$ means the Krull dimension of the ring $B/P$. Finally, if $S$ is a subring of the ring $B$ and $Q$ and $P$ are prime ideals of $B$ satisfying $S \cap Q = S \cap P$, then we say that $Q$ and $P$ are glued together in $S$.

The following theorem is the main result from \cite{Gluing2021}. As mentioned in the previous section, it will be used heavily in our proofs.

\begin{theorem}[\cite{Gluing2021}, Theorem 2.14]\label{theorem: GluingTheorem}
Let $(B,M)$ be a reduced local ring containing the rationals with $B/M$ uncountable and $|B| = |B/M|$.  
%Suppose $B$ contains the rationals and $\Min(B) = \{Q_1, Q_2, \ldots ,Q_n\}$ with $n \geq 2$, and 
Suppose $\Min(B)$ is partitioned into $m \geq 1$ subcollections $C_1, \ldots ,C_m$. Then there is a reduced local ring $S \subseteq B$ with maximal ideal $S \cap M$ such that 
\begin{enumerate}
\item $S$ contains the rationals, 
\item $\widehat{S} = \widehat{B}$, 
\item $S/(S \cap M)$ is uncountable and $|S| = |S/(S \cap M)|$, 
\item For $Q, Q' \in \Min(B)$, $Q \cap S = Q' \cap S$ if and only if there is an $i \in \{1,2, \ldots ,m\}$ with $Q \in C_i$ and $Q' \in C_i$, 
\item The map $f:\Spec(B) \longrightarrow \Spec(S)$ given by $f(P) = S \cap P$ is onto and, if $P$ is a prime ideal of $B$ with positive height, then $f(P)B = P$.  In particular, if $P$ and $P'$ are prime ideals of $B$ with positive height, then $f(P)$ has positive height and $f(P) = f(P')$ implies that $P = P'$. 
%\item $\{P \in \Spec(B) \, | \, P \not\in \Min(B)\}$ and $\{P \in \Spec(S) \, | \, P \not\in \Min(S)\}$ are isomorphic as posets.
\end{enumerate}
\end{theorem}

The conclusions of Theorem \ref{theorem: GluingTheorem} imply that, when viewed as partially ordered sets, the prime spectrum of $B$ and the prime spectrum of $S$ are the same except that some of the minimal prime ideals of $B$ are glued together in $S$.

As is indicated in the following proposition, the ring $S$ obtained in Theorem \ref{theorem: GluingTheorem} is quasi-excellent as long as the ring $B$ is quasi-excellent.

\begin{proposition}[\cite{UFDSpecs}, Proposition 3.6]\label{quasiexcellent}
Suppose the ring $B$ in Theorem \ref{theorem: GluingTheorem} satisfies the additional condition that it is quasi-excellent. Then the ring $S$ constructed in Theorem \ref{theorem: GluingTheorem} is also quasi-excellent.
\end{proposition}
    
Before we begin our construction, we present an example to build intuition for how our construction works. Given a specific complete local ring $T$ and two prime ideals $Q_1$ and $Q_2$ of $T$, we build a precompletion of $T$ where $Q_1$ and $Q_2$ are glued together and where the prime spectra of the two rings are related in a desired way.

    \begin{example}\label{mainexample}
        Let
        $T = \C[[x_1, \ldots, x_8]]/J$ where $$J = (x_1, x_5) \cap (x_1, x_6, x_7) \cap (x_2, x_3, x_5) \cap (x_2, x_3, x_6, x_7),$$ and let $Q_1 = (x_1, x_5, x_6, x_7, x_8)/J$ and $Q_2 = (x_2, x_3, x_5, x_6, x_7, x_8)/J$. Then $T$ is a complete local ring with a part of its prime spectrum pictured as a partially ordered set in Figure 1.  Suppose we want to construct a local subring $(B, B \cap M)$ of $T$ such that $\widehat{B} \cong T$ and such that $B \cap Q_1 = B \cap Q_2$. Moreover, the set of prime ideals of $T/(Q_1 \cap Q_2)$ of positive height and the set of prime ideals of $B/(B \cap Q_1) = B/(B \cap Q_2)$ of positive height are isomorphic as posets. In other words, the prime spectrum of $T$ above $Q_1$ and $Q_2$ and the prime spectrum of $B$ above $B \cap Q_1 = B \cap Q_2$ are isomorphic as posets. Part of the prime spectrum of such a ring $B$ is pictured in Figure 2. 
        \begin{figure}[h!]\label{SpecT}
            \centering
            \begin{minipage}[t]{0.55\textwidth}
                \vspace{0pt}
                \centering
                \begin{forest}
                    for tree={
                        edge+={thick},
                        inner sep=1.25pt,
                        circle, draw, fill,
                        l = 12mm,
                        s sep = 12mm
                    }
                    [, alias=M
                        [, alias=L1
                            [, alias=L2
                                [, alias=L3
                                    [, alias=LL1
                                        [, alias=LL2
                                            [, alias=LL3]
                                        ]
                                    ]
                                    [, alias=LR1
                                        [, alias=LR2]
                                    ]
                                ]
                            ]
                        ]
                        [, alias=R2
                            [, alias=R3
                                [, alias=RL1
                                    [, alias=RL2
                                        [, alias=RL3]
                                    ]
                                ]
                                [, alias=RR1
                                    [, alias=RR2]
                                ]
                            ]
                        ]
                    ]
                    \fill [blue, opacity=0.15, rounded corners] (-2.5, -4.23) rectangle (-1.0, -2.88);
                    \fill [blue, opacity=0.15, rounded corners] (1.0, -3.0) rectangle (2.5, -1.75);
                    \node at (-2, -3.5) {$Q_1$};
                    \node at (2, -2.25) {$Q_2$};
                \end{forest}
                \caption{Partial $\Spec (T)$}
            \end{minipage}
            \begin{minipage}[t]{0.44\textwidth}
                \vspace{25pt}
                \centering
                \begin{tikzpicture}
                    \node (M)[circle, fill, inner sep=1.5pt] {};
                    \node (L1)[circle, fill, inner sep=1.5pt] [below left of=M] {};
                    \node (L2)[circle, fill, inner sep=1.5pt] [below of=L1] {};
                    \node (L3)[circle, fill, inner sep=1.5pt] [below right of=L2] {};
                    \node (R1)[circle, fill, inner sep=1.5pt] [below right of=M, yshift=-0.5cm] {};
                    \node (LL1)[circle, fill, inner sep=1.5pt] [below left of=L3] {};
                    \node (LL2)[circle, fill, inner sep=1.5pt] [below of=LL1] {};
                    \node (LL3)[circle, fill, inner sep=1.5pt] [below of=LL2] {};
                    \node (LR1)[circle, fill, inner sep=1.5pt, yshift=0.5cm] [below right of=L3, yshift=-0.5cm] {};
                    \node (LR2)[circle, fill, inner sep=1.5pt, yshift=0.5cm] [below of=LR1, yshift=-0.5cm] {};
    
                    \draw[line width=0.25mm] (M) -- (L1) -- (L2) -- (L3) -- (LL1) -- (LL2) -- (LL3);
                    \draw[line width=0.25mm] (M) -- (R1) -- (L3) -- (LR1) -- (LR2);
                    % \draw[line width=0.25mm] (M) -- (R1) -- (L3);
                \end{tikzpicture}
                \caption{Partial $\Spec (B)$}
            \end{minipage}
        \end{figure}

        \noindent To find such a ring $B$, first define
        \begin{equation*}
            \overline{T} = \frac{T}{Q_1 \cap Q_2} \cong \frac{\C[[x_1, x_2,x_3, x_4]]}{(x_1) \cap (x_2, x_3)}
        \end{equation*}
        and note that
        \begin{equation*}
            \frac{\overline{T}[[y_1, y_2,y_3, y_4]]}{(y_1) \cap (y_2, y_3)} \cong \frac{\C[[x_1, x_2,x_3, x_4, y_1, y_2,y_3, y_4]]}{(x_1, y_1) \cap (x_1, y_2, y_3) \cap (x_2, x_3, y_1) \cap (x_2, x_3, y_2, y_3)},
        \end{equation*}
        which is naturally isomorphic to $T$ by mapping $y_i \to x_{i+4}$ for $i = 1,2,3,4$. Denote the minimal prime ideals of $\overline{T}$ by $\overline{Q}_1 = (x_1)/((x_1) \cap (x_2,x_3))$ and $\overline{Q}_2 = (x_2,x_3)/((x_1) \cap (x_2,x_3))$. Use Theorem \ref{theorem: GluingTheorem} to find a local subring $R$ of $\overline{T}$ such that $\widehat{R} \cong \overline{T}$, $R \cap \overline{Q}_1 = R \cap \overline{Q}_2 = (0)$, and the parts of the prime spectra of $\overline{T}$ and $R$ of positive height are isomorphic as posets. Then setting $B = R[[y_1, y_2,y_3, y_4]]/((y_1) \cap (y_2, y_3))$, we have
        \begin{align*}
            \widehat{B} &= \widehat{\frac{R[[y_1, y_2,y_3, y_4]]}{(y_1) \cap (y_2, y_3)}}\\
            &\cong \frac{\widehat{R}[[y_1, y_2,y_3, y_4]]}{(y_1) \cap (y_2, y_3)}\\
            &\cong T.
        \end{align*}
    Now, letting $\overline{Q}^e_1$ denote the ideal $\overline{Q}_1\overline{T}[[y_1,y_2,y_3,y_4]]$ and $\overline{Q}^e_2$ denote the ideal $\overline{Q}_2\overline{T}[[y_1,y_2,y_3,y_4]]$, we have $$R[[y_1, y_2,y_3 ,y_4]] \cap (\overline{Q}^e_1,y_1, y_2,y_3 ,y_4) = R[[y_1, y_2,y_3 ,y_4]] \cap (\overline{Q}^e_2,y_1, y_2,y_3 ,y_4) = (y_1, y_2,y_3 ,y_4),$$ and it follows that $B \cap Q_1 = B \cap Q_2$. Finally, $T/(Q_1 \cap Q_2) \cong \C[[x_1, x_2,x_3 ,x_4]]/((x_1) \cap (x_2,x_3)) = \overline{T}$ and $B/(B \cap Q_1) = B/(B \cap Q_2) \cong R.$ Therefore, the set of prime ideals of $T/(Q_1 \cap Q_2)$ of positive height and the set of prime ideals of $B/(B \cap Q_1) = B/(B \cap Q_2)$ of positive height are isomorphic as posets. 
   %     We show that, by the way $R$ was constructed, $B$ is our desired subring of $T$. \textcolor{red}{Put more here to explain why the other properties hold!!}
    %    Then there does exist a precompletion $(B, B \cap M)$ of $T$ with the partial prime spectrum envisioned through the Hasse diagram on the top right.
    Note that, since $T$ is excellent, $R$ is quasi-excellent by Proposition \ref{quasiexcellent}. Since power series rings over a quasi-excellent ring and quotient rings of a quasi-excellent ring are also quasi-excellent, the ring $B$ is quasi-excellent.
    \end{example}

    Intuitively, this will be our approach: to glue together two prime ideals of $T$, call them $Q_1$ and $Q_2$, we consider $\overline{T} = T/(Q_1 \cap Q_2)$ and use Theorem \ref{theorem: GluingTheorem} to find a subring of $\overline{T}$ where the minimal prime ideals of $\overline{T}$ are glued together. We then find a way to \textit{return} to $T$ from $\overline{T}$ by adjoining indeterminates and then modding out by an ideal, and then demonstrate that there is a precompletion $B$ of $T$ with $Q_1$ and $Q_2$ glued together. In addition, we will argue that, as is true for Example \ref{mainexample}, there is a coheight preserving bijection between the minimal prime ideals of $T_{Q_i}$ and the minimal prime ideals of $B_{B \cap Q_1}$ for $i = 1$ and $i = 2$.

Note that the ring $B$ from Example \ref{mainexample} satisfies the conditions of Theorem \ref{theorem: GluingTheorem}, and so we can apply it to $B$ to obtain a local subring of $B$ where the two minimal prime ideals of $B$ are glued together. The prime spectrum of the resulting ring will contain a poset similar to Figure 2 except that the two minimal elements are glued together. So, for this $T$ we have answered Question \ref{question: easier classication question} for two different partially ordered sets; the one pictured in Figure 2 and the one obtained from Figure 2 by gluing the two minimal elements.

We end this section with a useful proposition from \cite{Ostermeyer}.

\begin{proposition}
        [\cite{Ostermeyer}, Lemma 3.2]\label{proposition: folklore going down}
        Let $(T, M)$ be a local ring and suppose $(B, B \cap M)$ is a local subring of $T$ such that $\widehat{T} \cong \widehat{B} $.  Then the going down property holds between $T$ and $B$.
    \end{proposition}

   % While we will not need it, I suspect it is not difficult to show that the $P$ of Corollary \ref{corollary: minimal prime properties} is in fact a minimal prime ideal over $I_i$. 

  \section{The Main Result}\label{MainResult}

  In this section, we prove our main result. To do so, we work with a local ring $(T, M)$ and two prime ideals $Q_1$ and $Q_2$ of $T$.  We define $\overline{T} = T / (Q_1 \cap Q_2)$ and for notational convenience, if $I$ is an ideal of $T$ with $Q_1 \cap Q_2 \subseteq I$, then we use $\overline{I}$ to denote $ I / (Q_1 \cap Q_2) \subseteq \overline{T}$.
In addition, we use the following convention: if $R$ is a Noetherian ring, $R[[y_1, \ldots, y_n]]$ is the ring of formal power series over $R$, and $I$ is an ideal of $R$, then we define $I^{e} = IR[[y_1, \ldots, y_n]]$.  
    %Think of this as ``blowing up'' $I$ by extending it to the ring $R[[y_1, \ldots, y_n]]$.
    
 The following lemma is instrumental for our construction.

 \begin{lemma}
\label{lemma: two to one correspondence alt}
Let $y_1, \ldots, y_n$ be indeterminates and let $I$ be a proper ideal of the ring $\Q[[y_1, \ldots, y_n]]$. Suppose $\{\mathfrak{q}_1, \ldots , \mathfrak{q}_k\}$ is the set of minimal prime ideals over $I$ in $\Q[[y_1, \ldots ,y_n]].$
Assume that if $D$ is a domain containing the rationals then, for all $i = 1,2, \ldots ,k$, $\mathfrak{q}_iD[[y_1, \ldots ,y_n]]$ is a prime ideal of $D[[y_1, \ldots ,y_n]]$ and $\mathfrak{q}_iD[[y_1, \ldots ,y_n]] \subseteq \mathfrak{q}_jD[[y_1, \ldots ,y_n]]$ if and only if $i = j$.

%\begin{equation*}
 %           T \cong \frac{\overline{T}[[y_1, \ldots, y_n]]}{I\overline{T}[[y_1, \ldots, y_n]]}.
%        \end{equation*}
%
Let $(T, M)$ be a local ring containing the rationals  and let $Q_1, Q_2 \in \Spec (T)$ with $Q_1 \not\subseteq Q_2$ and $Q_2 \not\subseteq Q_1$. Let $\overline{T} = T / (Q_1 \cap Q_2)$ and let $(R, R \cap \overline{M})$ be a local subring of $(\overline{T}, \overline{M})$ such that $R$ contains the rationals, $R$ is a domain, and $\widehat{R} \cong \widehat{\overline{T}}$.  Then the minimal prime ideals over $I\overline{T}[[y_1, \ldots, y_n]]$ in $\overline{T}[[y_1, \ldots, y_n]]$ are given by 
        \begin{align*}
            (\overline{Q}_1^e, \mathfrak{q}_1\overline{T}[[y_1,\ldots ,y_n]]), &\ldots, (\overline{Q}_1^e, \mathfrak{q}_k\overline{T}[[y_1,\ldots ,y_n]]), \\
            (\overline{Q}_2^e, \mathfrak{q}_1\overline{T}[[y_1,\ldots ,y_n]]), &\ldots, (\overline{Q}_2^e, \mathfrak{q}_k\overline{T}[[y_1,\ldots ,y_n]]),
        \end{align*}
        and the minimal prime ideals over $IR[[y_1, \ldots, y_n]]$ in $R[[y_1, \ldots, y_n]]$ are given by 
        \begin{equation*}
            \mathfrak{q}_1R[[y_1, \ldots, y_n]], \ldots, \mathfrak{q}_kR[[y_1, \ldots, y_n]].
        \end{equation*}
    \end{lemma}

\begin{proof}
Since $T$ contains the rationals, so does $\overline{T}, \overline{T}/\overline{Q}_1$, and $\overline{T}/\overline{Q}_2$.  For $j = 1,2, \ldots,k$ and for $i = 1,2$, we have
%\begin{equation}
%\label{equation: prime proof}
$$\frac{\overline{T}/\overline{Q}_i[[y_1, \ldots, y_n]]}{\mathfrak{q}_j\overline{T}/\overline{Q}_i[[y_1, \ldots, y_n]]} \cong \frac{\overline{T}[[y_1, \ldots, y_n]]/\overline{Q}^e_i}{(\overline{Q}^e_i, \mathfrak{q}_j\overline{T}[[y_1, \ldots ,y_n]])/\overline{Q}^{e}_i} \cong \frac{\overline{T}[[y_1, \ldots, y_n]]}{(\overline{Q}_i^e, \mathfrak{q}_j\overline{T}[[y_1, \ldots ,y_n]])}.$$
By assumption, $\mathfrak{q}_j\overline{T}/\overline{Q}_i[[y_1, \ldots, y_n]]$ is a prime ideal of $\overline{T}/\overline{Q}_i[[y_1, \ldots, y_n]]$ and so $(\overline{Q}_i^e, \mathfrak{q}_j\overline{T}[[y_1, \ldots ,y_n]])$ is a prime ideal of $\overline{T}[[y_1, \ldots ,y_n]]$. 

%We now show that 
%$$(\overline{Q}_1^e, \mathfrak{q}_1\overline{T}[[y_1,\ldots ,y_n]]) \cap \cdots \cap (\overline{Q}_1^e, \mathfrak{q}_k\overline{T}[[y_1,\ldots ,y_n]]) \cap $$
%$$(\overline{Q}_2^e, \mathfrak{q}_1\overline{T}[[y_1,\ldots ,y_n]]) \cap \cdots \cap (\overline{Q}_2^e, \mathfrak{q}_k\overline{T}[[y_1,\ldots ,y_n]])$$ is a minimal primary decomposition of $I\overline{T}[[y_1, \ldots, y_n]]$. Since $I\overline{T}[[y_1, \ldots, y_n]]$ is a radical ideal of $\overline{T}[[y_1, \ldots, y_n]]$, it will then follow that the minimal prime ideals over $I\overline{T}[[y_1, \ldots, y_n]]$ in $\overline{T}[[y_1, \ldots, y_n]]$ are given by 
%        \begin{equation*}
%            (\overline{Q}_1^e, \mathfrak{q}_1\overline{T}[[y_1,\ldots ,y_n]]), \ldots, (\overline{Q}_1^e, \mathfrak{q}_k\overline{T}[[y_1,\ldots ,y_n]]), (\overline{Q}_2^e, \mathfrak{q}_1\overline{T}[[y_1,\ldots ,y_n]]), \ldots, (\overline{Q}_2^e, \mathfrak{q}_k\overline{T}[[y_1,\ldots ,y_n]]).
%        \end{equation*}
%\end{equation} 

Let $J$ be a prime ideal of $\overline{T}[[y_1, \ldots ,y_n]]$ such that $I\overline{T}[[y_1, \ldots ,y_n]] \subseteq J.$ Then
$$I \subseteq I\overline{T}[[y_1, \ldots ,y_n]] \cap \Q[[y_1, \ldots ,y_n]]\subseteq J \cap \Q[[y_1, \ldots ,y_n]],$$ and so $J \cap \Q[[y_1, \ldots ,y_n]]$ is a prime ideal of $\Q[[y_1, \ldots ,y_n]]$ containing $I$. It follows that $J \cap \Q[[y_1, \ldots ,y_n]]$ contains a minimal prime ideal of $I$. Thus, $J \cap \Q[[y_1, \ldots ,y_n]]$ contains $\mathfrak{q}_{\ell}$ for some $\ell \in \{1,2, \ldots ,k\}$, and we have $\mathfrak{q}_{\ell}\overline{T}[[y_1, \ldots ,y_n]] \subseteq J.$ The minimal prime ideals of $\overline{T}$ are $\overline{Q}_1$ and $\overline{Q}_2.$ Hence, the minimal prime ideals of $\overline{T}[[y_1, \ldots ,y_n]]$ are $\overline{Q}^e_1$ and $\overline{Q}^e_2.$ As $J$ must contain a minimal prime ideal of $\overline{T}[[y_1, \ldots ,y_n]]$, we have that $J$ contains $\overline{Q}^e_1$ or $\overline{Q}^e_2.$ Now, $I\overline{T}[[y_1, \ldots ,y_n]] \subseteq (\overline{Q}_{i}^e, \mathfrak{q}_{j}\overline{T}[[y_1,\ldots ,y_n]])$ for $i = 1,2$ and for all $j = 1,2, \ldots ,k$ and $(\overline{Q}_{i}^e,\mathfrak{q}_{\ell}\overline{T}[[y_1,\ldots ,y_n]]) \subseteq J$ for $i = 1$ or $i = 2$. It follows that if $J$ is a minimal prime ideal of $I\overline{T}[[y_1, \ldots ,y_n]]$, then $J = (\overline{Q}_i^e, \mathfrak{q}_j\overline{T}[[y_1,\ldots ,y_n]])$ for some $i = 1,2$ and for some $j = 1,2, \ldots ,k$.

Let $\mu \in \{1,2\}$ and let $\rho \in \{1,2, \ldots ,k\}$. The prime ideal $(\overline{Q}_{\mu}^e, \mathfrak{q}_{\rho}\overline{T}[[y_1,\ldots ,y_n]])$ contains $I\overline{T}[[y_1,\ldots ,y_n]]$ and so it contains a minimal prime ideal over $I\overline{T}[[y_1,\ldots ,y_n]]$. By what was shown in the previous paragraph, we have that $(\overline{Q}_i^e, \mathfrak{q}_j\overline{T}[[y_1,\ldots ,y_n]]) \subseteq (\overline{Q}_{\mu}^e, \mathfrak{q}_{\rho}\overline{T}[[y_1,\ldots ,y_n]])$ for some $i \in \{1,2\}$ and some $j \in \{1,2, \ldots ,k\}$ where $(\overline{Q}_i^e, \mathfrak{q}_j\overline{T}[[y_1,\ldots ,y_n]])$ is a minimal prime ideal over $I\overline{T}[[y_1,\ldots ,y_n]]$. If $i \neq \mu$, then there exists $z \in \overline{Q}_i$ with $z \not\in \overline{Q}_{\mu}.$ The elements of $\mathfrak{q}_{\rho}\overline{T}[[y_1,\ldots ,y_n]]$ all have constant term equal to zero, and so $z \in (\overline{Q}_i^e, \mathfrak{q}_j\overline{T}[[y_1,\ldots ,y_n]])$ but $z \not\in (\overline{Q}_{\mu}^e, \mathfrak{q}_{\rho}\overline{T}[[y_1,\ldots ,y_n]]),$ a contradiction. It follows that $i = \mu$. Thus
$$\frac{(\overline{Q}_{\mu}^e,\mathfrak{q}_{j}\overline{T}[[y_1,\ldots ,y_n]])}{\overline{Q}_{\mu}^e }\subseteq \frac{(\overline{Q}_{\mu}^e, \mathfrak{q}_{\rho}\overline{T}[[y_1,\ldots ,y_n]])}{\overline{Q}_{\mu}^e}$$ and so $\mathfrak{q}_{j}\overline{T}/\overline{Q}_{\mu}[[y_1,\ldots ,y_n]] \subseteq \mathfrak{q}_{\rho}\overline{T}/\overline{Q}_{\mu}[[y_1,\ldots ,y_n]]$. By hypotheses, $j = \rho$ and we have that $(\overline{Q}_{\mu}^e, \mathfrak{q}_{\rho}\overline{T}[[y_1,\ldots ,y_n]])$ is a minimal prime ideal over $I\overline{T}[[y_1,\ldots ,y_n]]$. Therefore the minimal prime ideals over $I\overline{T}[[y_1, \ldots, y_n]]$ in $\overline{T}[[y_1, \ldots, y_n]]$ are given by 
        \begin{equation*}
            (\overline{Q}_1^e, \mathfrak{q}_1\overline{T}[[y_1,\ldots ,y_n]]), \ldots, (\overline{Q}_1^e, \mathfrak{q}_k\overline{T}[[y_1,\ldots ,y_n]]), (\overline{Q}_2^e, \mathfrak{q}_1\overline{T}[[y_1,\ldots ,y_n]]), \ldots, (\overline{Q}_2^e, \mathfrak{q}_k\overline{T}[[y_1,\ldots ,y_n]]).
        \end{equation*}

By assumption, $\mathfrak{q}_j R[[y_1, \ldots ,y_n]]$ is a prime ideal of $R[[y_1, \ldots ,y_n]]$ for all $j = 1,2, \ldots ,k,$ and so $\mathfrak{q}_j R[[y_1, \ldots ,y_n]]/IR[[y_1, \ldots ,y_n]]$ is a prime ideal of $R[[y_1, \ldots ,y_n]]/IR[[y_1, \ldots ,y_n]]$ for all $j = 1,2, \ldots ,k$. Note that the completion of $R[[y_1, \ldots ,y_n]]$ is isomorphic to the completion of $\overline{T}[[y_1, \ldots ,y_n]]$. 
 It follows that the completion of $R[[y_1, \ldots ,y_n]]/IR[[y_1, \ldots ,y_n]]$ is isomorphic to the completion of $\overline{T}[[y_1, \ldots ,y_n]]/I\overline{T}[[y_1, \ldots ,y_n]]$. By Proposition \ref{proposition: folklore going down}, the intersection of a minimal prime ideal of $\overline{T}[[y_1, \ldots ,y_n]]/I\overline{T}[[y_1, \ldots ,y_n]]$ with $R[[y_1, \ldots ,y_n]]/IR[[y_1, \ldots ,y_n]]$ is a minimal prime ideal of $R[[y_1, \ldots ,y_n]]/IR[[y_1, \ldots ,y_n]]$. Let $\mathcal{J} =(\overline{Q}_i^e, \mathfrak{q}_j\overline{T}[[y_1,\ldots ,y_n]])/I\overline{T}[[y_1, \ldots ,y_n]]$ be a minimal prime ideal of $\overline{T}[[y_1, \ldots ,y_n]]/I\overline{T}[[y_1, \ldots ,y_n]]$. Then $\mathcal{J} \cap (R[[y_1, \ldots ,y_n]]/IR[[y_1, \ldots ,y_n]])$ contains the prime ideal $\mathfrak{q}_j R[[y_1, \ldots ,y_n]]/IR[[y_1, \ldots ,y_n]]$ of $ R[[y_1, \ldots ,y_n]]/IR[[y_1, \ldots ,y_n]]$. Therefore, $\mathcal{J} \cap (R[[y_1, \ldots ,y_n]]/IR[[y_1, \ldots ,y_n]]) = \mathfrak{q}_j R[[y_1, \ldots ,y_n]]/IR[[y_1, \ldots ,y_n]]$. Note that the map from the set of prime ideals of $\overline{T}[[y_1, \ldots ,y_n]]/I\overline{T}[[y_1, \ldots ,y_n]]$ to the set of prime ideals of $R[[y_1, \ldots ,y_n]]/IR[[y_1, \ldots ,y_n]]$ given by intersection is onto. Thus, the set of minimal prime ideals of $R[[y_1, \ldots ,y_n]]/IR[[y_1, \ldots ,y_n]]$ is 
$$\{\mathfrak{q}_1R[[y_1, \ldots ,y_n]]/IR[[y_1, \ldots ,y_n]], \ldots ,\mathfrak{q}_kR[[y_1, \ldots ,y_n]]/IR[[y_1, \ldots ,y_n]]\},$$
and it follows that the minimal prime ideals over $IR[[y_1, \ldots ,y_n]]$ in $R[[y_1, \ldots ,y_n]]$ are given by $$\mathfrak{q}_1R[[y_1, \ldots ,y_n]], \ldots ,\mathfrak{q}_kR[[y_1, \ldots ,y_n]].$$
\end{proof}

    \begin{remark}
        \label{remark: two to one correspondence}
        Working in the context of Lemma \ref{lemma: two to one correspondence alt}, there is a two-to-one correspondence between the minimal prime ideals of $\overline{T}[[y_1, \ldots, y_n]]/I\overline{T}[[y_1, \ldots, y_n]]$ and the minimal prime ideals of $R[[y_1, \ldots, y_n]]/IR[[y_1, \ldots, y_n]]$, and that correspondence is given by intersection.  Moreover, for $i \in \{1,2\}$, there is a bijection between minimal prime ideals of $\overline{T}[[y_1, \ldots, y_n]]/I\overline{T}[[y_1, \ldots, y_n]]$ containing  $\overline{Q}^e_i/I\overline{T}[[y_1, \ldots, y_n]]$ and minimal prime ideals of $R[[y_1, \ldots, y_n]]/IR[[y_1, \ldots, y_n]]$.  Again, this bijection is given by intersection.
    \end{remark}

    In Theorem \ref{theorem: generalized gluing alt}, we use Remark \ref{remark: two to one correspondence} to demonstrate that, for the rings $T$ and $B$ from Lemma \ref{lemma: two to one correspondence alt},  there is a  coheight preserving bijection between the minimal prime ideals of $T_{Q_i}$ and the minimal prime ideals of $B_{B \cap Q_1}$ where $i \in \{1,2\}$. This shows that, given a saturated chain of prime ideals of $T$ that starts at a minimal prime ideal and ends at $Q_i$ and has maximal length, there is a corresponding saturated chain of prime ideals of $B$ of the same length that starts at a minimal prime ideal of $B$ and ends at $B \cap Q_i$. In other words, there are important parts of the prime spectrum of $T$ below both $Q_1$ and $Q_2$ that exist in the prime spectrum of $B$ below $B \cap Q_1$.

    We are now ready to state the main result of this paper.

\begin{theorem}
\label{theorem: generalized gluing alt}
Let $y_1, \ldots, y_n$ be indeterminates and let $I$ be a proper ideal of the ring $\Q[[y_1, \ldots, y_n]]$. Suppose $\{\mathfrak{q}_1, \ldots , \mathfrak{q}_k\}$ is the set of minimal prime ideals over $I$ in $\Q[[y_1, \ldots ,y_n]].$
Assume that if $D$ is a domain containing the rationals then, for all $i = 1,2, \ldots ,k$, $\mathfrak{q}_iD[[y_1, \ldots ,y_n]]$ is a prime ideal of $D[[y_1, \ldots ,y_n]]$ and $\mathfrak{q}_iD[[y_1, \ldots ,y_n]] \subseteq \mathfrak{q}_jD[[y_1, \ldots ,y_n]]$ if and only if $i = j$.

Let $(T, M)$ be an uncountable local ring containing the rationals such that $\abs{T} = \abs{T/M}$. Let $Q_1, Q_2 \in \Spec (T)$ be such that $Q_1 \not\subseteq Q_2$ and $Q_2 \not\subseteq Q_1$ and such that $T_{Q_1}$ and $T_{Q_2}$ are catenary.  Let $\overline{T} = T / (Q_1 \cap Q_2)$ and assume that there is a ring isomorphism
\begin{equation*}
    \phi: T \longrightarrow   \frac{\overline{T}[[y_1, \ldots, y_n]]}{I\overline{T}[[y_1, \ldots, y_n]]}
\end{equation*}
   satisfying $\phi(Q_1) = (\overline{Q}^e_1, y_1, \ldots ,y_n)/I\overline{T}[[y_1, \ldots, y_n]]$ and $\phi(Q_2) = (\overline{Q}^e_2, y_1, \ldots ,y_n)/I\overline{T}[[y_1, \ldots, y_n]]$. Then there exists an uncountable local subring $(B, B \cap M)$ of $T$ containing the rationals that satisfies the following conditions:
    
\begin{enumerate}
    \item $\widehat{B} \cong \widehat{T}$,
    \item $\abs{B} = \abs{B / (B \cap M)}$,
    \item $B \cap Q_1 = B \cap Q_2$,
    \item $B_{B \cap Q_1}$ is catenary,
    \item $B \cong R[[y_1, \ldots, y_n]]/IR[[y_1, \ldots, y_n]]$ where $R$ is a domain containing the rationals,
    \item The map $f: \Spec (T/(Q_1 \cap Q_2)) \to \Spec (B/(B \cap Q_1))$ given by $f(P/(Q_1 \cap Q_2)) = (B \cap P)/(B \cap Q_1)$ is a bijection when restricted to prime ideals of positive height,  
    %Moreover, for the two minimal prime ideals $\overline{Q}_1, \overline{Q}_2 \in \Min \overline{T}$, we have $f(\overline{Q}_1) = f(\overline{Q}_2) = (0)_{\overline{B}}$,
    \item For $i \in \{1, 2\}$, there is a coheight preserving bijection between $\Min T_{Q_i}$ and $\Min B_{B \cap Q_1}$,
    \item If $T$ is quasi-excellent, then so is $B$.
        \end{enumerate}
    \end{theorem}

\begin{proof}
The local ring $\overline{T}$ satisfies the conditions of Theorem \ref{theorem: GluingTheorem} and so there is an uncountable local subring $(R, R \cap \overline{M})$ of $\overline{T}$ such that $R$ contains the rationals, $\widehat{R} \cong \widehat{\overline{T}}$, $|R| = |R/(R \cap \overline{M})|$, $R \cap \overline{Q}_1 = R \cap \overline{Q}_2$, and the map $g: \Spec (\overline{T}) \longrightarrow \Spec(R)$ given by $g(P) = R \cap P$ is a bijection when restricted to prime ideals of postive height. Note that $R$ is reduced and has only one minimal prime ideal, and so $R$ is a domain 
and we have $R \cap \overline{Q}_1 = R \cap \overline{Q}_2 = (0)$. 
Let 
$$    B' = \frac{R[[y_1, \ldots, y_n]]}{IR[[y_1, \ldots, y_n]]} \mbox{ and let } T' = \frac{\overline{T}[[y_1, \ldots, y_n]]}{I\overline{T}[[y_1, \ldots, y_n]]}.$$
Then
    \begin{align*}
    \widehat{B'} \cong \frac{\widehat{R}[[y_1, \ldots, y_n]]}{I\widehat{R}[[y_1, \ldots, y_n]]}
%    &\cong \frac{\widehat{R[[y_1, \ldots, y_n]]}}{(IR[[y_1, \ldots, y_n]])\widehat{R[[y_1, \ldots, y_n]]}}\\
    \cong \frac{\widehat{\overline{T}}[[y_1, \ldots, y_n]]}{I\widehat{\overline{T}}[[y_1, \ldots, y_n]]} 
    &\cong \frac{\widehat{\overline{T}[[y_1, \ldots, y_n]]}}{I\overline{T}[[y_1, \ldots, y_n]]} \cong \widehat{T'} \cong \widehat{T}.
        \end{align*}
    Since $R$ contains the rationals, so does $B'$. Observe that $M_{B'} = (R \cap \overline{M},y_1, \ldots ,y_n)/IR[[y_1, \ldots, y_n]]$ is the maximal ideal of $B'$. Now $|B'| \leq |T'| = |T| = |T/M| = |\overline{T}/\overline{M}| = |R/(R \cap \overline{M})| = |B'/M_{B'}|$, and it follows that $|B'| = |B'/M_{B'}|$ and that $B'$ is uncountable.

 Since $R \cap \overline{Q}_1 = R \cap \overline{Q}_2 = (0)$, we have that $$(\overline{Q}^e_1, y_1, \ldots ,y_n) \cap R[[y_1, \ldots ,y_n]] = (\overline{Q}^e_2, y_1, \ldots ,y_n) \cap R[[y_1, \ldots ,y_n]] = (y_1, \ldots ,y_n)R[[y_1, \ldots ,y_n]]$$ and so letting $Q'_1 = (\overline{Q}^e_1, y_1, \ldots ,y_n)/I\overline{T}[[y_1, \ldots, y_n]]$ and $Q'_2 = (\overline{Q}^e_2, y_1, \ldots ,y_n)/I\overline{T}[[y_1, \ldots, y_n]]$ we have $$B' \cap Q'_1  = B' \cap Q'_2 = \frac{(y_1, \ldots ,y_n)R[[y_1, \ldots ,y_n]]}{IR[[y_1, \ldots ,y_n]]}.$$

Since $R$ is a domain, $$(0) \subsetneq (y_1) \subsetneq (y_1, y_2) \subsetneq \cdots \subsetneq (y_1, \ldots ,y_n)$$ is a chain of prime ideals of $R[[y_1, \ldots ,y_n]]$ of length $n$. By the generalized principal ideal theorem, the dimension of $R[[y_1, \ldots ,y_n]]_{(y_1, \ldots ,y_n)}$ is $n$. As the maximal ideal of $R[[y_1, \ldots ,y_n]]_{(y_1, \ldots ,y_n)}$ is generated by $n$ elements, the ring $R[[y_1, \ldots ,y_n]]_{(y_1, \ldots ,y_n)}$ is a regular local ring and thus is catenary. It follows that $B'_{B' \cap Q'_1}$ is catenary.

 Now, $$\frac{T'}{Q'_1 \cap Q'_2} \cong \overline{T}$$ and $$\frac{B'}{B' \cap Q'_1} \cong R.$$ Since $g$ is a bijection when restricted to prime ideals of positive height, so is the map $$f' : \Spec\left(\frac{T'}{Q'_1 \cap Q'_2}\right) \longrightarrow \Spec\left(\frac{B'}{B' \cap Q'_1}\right)$$ given by intersection.

 By Lemma \ref{lemma: two to one correspondence alt}, the minimal prime ideals of $T'$ are given by 
        \begin{equation*}
            \frac{(\overline{Q}_1^e, \mathfrak{q}_1\overline{T}[[y_1,\ldots ,y_n]])}{I\overline{T}[[y_1,\ldots ,y_n]]}, \ldots, \frac{(\overline{Q}_1^e, \mathfrak{q}_k\overline{T}[[y_1,\ldots ,y_n]])}{I\overline{T}[[y_1,\ldots ,y_n]]}, \frac{(\overline{Q}_2^e, \mathfrak{q}_1\overline{T}[[y_1,\ldots ,y_n]])}{I\overline{T}[[y_1,\ldots ,y_n]]}, \ldots, \frac{(\overline{Q}_2^e, \mathfrak{q}_k\overline{T}[[y_1,\ldots ,y_n]])}{I\overline{T}[[y_1,\ldots ,y_n]]},
        \end{equation*}
    and the minimal prime ideals of $B'$ are given by 
        \begin{equation*}
            \frac{\mathfrak{q}_1R[[y_1, \ldots, y_n]]}{IR[[y_1, \ldots, y_n]]}, \ldots, \frac{\mathfrak{q}_kR[[y_1, \ldots, y_n]]}{IR[[y_1, \ldots, y_n]]}.
        \end{equation*}
Since $\mathfrak{q}_i$ is an ideal of $\Q[[y_1, \ldots y_n]]$, $$\frac{\mathfrak{q}_iR[[y_1, \ldots, y_n]]}{IR[[y_1, \ldots, y_n]]} \subseteq \frac{(y_1, \ldots ,y_n)R[[y_1, \ldots ,y_n]]}{IR[[y_1, \ldots ,y_n]]}$$ for all $i = 1,2, \ldots ,k$. By Remark \ref{remark: two to one correspondence}, there is a bijection between the minimal prime ideals of $T'$ contained in $Q'_1$ and the minimal prime ideals of $B'$ contained in $B' 
\cap Q'_1$ where the ideal $(\overline{Q}_1^e, \mathfrak{q}_i\overline{T}[[y_1,\ldots ,y_n]])/I\overline{T}[[y_1,\ldots ,y_n]]$ corresponds to the ideal $\mathfrak{q}_iR[[y_1, \ldots, y_n]]/IR[[y_1, \ldots, y_n]]$.  In other words, there is a bijection between $\Min T'_{Q'_1}$ and $\Min B'_{B' \cap Q'_1}$. We show that the bijection is coheight preserving. First note that, because $T_{Q_1}$ is catenary, the isomorphism $\phi$ gives that $T'_{Q'_1}$ is catenary.

Now fix $i \in \{1,2, \ldots ,k\}$.  If $$\frac{\mathfrak{q}_iR[[y_1, \ldots, y_n]]}{IR[[y_1, \ldots, y_n]]} = \frac{(y_1, \ldots ,y_n)R[[y_1, \ldots, y_n]]}{IR[[y_1, \ldots, y_n]]}$$ then $$\frac{(\overline{Q}_1^e, \mathfrak{q}_i\overline{T}[[y_1,\ldots ,y_n]])}{I\overline{T}[[y_1,\ldots ,y_n]]} = Q'_1.$$
If $$\frac{\mathfrak{q}_iR[[y_1, \ldots, y_n]]}{IR[[y_1, \ldots, y_n]]} \neq \frac{(y_1, \ldots ,y_n)R[[y_1, \ldots, y_n]]}{IR[[y_1, \ldots, y_n]]}$$ then let $z_1 \in (y_1, \ldots ,y_n)R[[y_1, \ldots, y_n]]/IR[[y_1, \ldots, y_n]]$ with $z_1 \not\in \mathfrak{q}_iR[[y_1, \ldots, y_n]]/IR[[y_1, \ldots, y_n]]$.
Let $J_0 = (\overline{Q}_1^e, \mathfrak{q}_i\overline{T}[[y_1,\ldots ,y_n]])/I\overline{T}[[y_1,\ldots ,y_n]]$ and let $J_1$ be a minimal prime ideal over $(J_0, z_1)$ contained in $Q'_1$. Then in the domain $T'/J_0$ the ideal $(J_0,z_1)/J_0$ is principal, and so $J_1/J_0$ is a height one prime ideal. It follows that the chain $J_0 \subsetneq J_1$ is saturated. Note that $B' \cap J_0 \subsetneq B' \cap J_1$. 

If $B' \cap J_1 = (y_1, \ldots ,y_n)R[[y_1, \ldots, y_n]]/IR[[y_1, \ldots, y_n]]$ then $Q'_1 = (B' \cap J_1)T' \subseteq J_1$ and so $J_1 = Q'_1$. If $B' \cap J_1 \neq (y_1, \ldots ,y_n)R[[y_1, \ldots, y_n]]/IR[[y_1, \ldots, y_n]]$ then let $$z_2 \in \frac{(y_1, \ldots ,y_n)R[[y_1, \ldots, y_n]]}{IR[[y_1, \ldots, y_n]]}$$ with $z_2 \not\in B' \cap J_1$. Let $J_2$ be a minimal prime ideal over $(J_1, z_2)$ contained in $Q'_1$. Then in the domain $T'/J_1$ the ideal $(J_1,z_2)/J_1$ is principal, and so $J_2/J_1$ is a height one prime ideal. It follows that the chain $J_1 \subsetneq J_2$ is saturated. Note that $B' \cap J_1 \subsetneq B' \cap J_2$. 

Continue this process to find a saturated chain of prime ideals of $T'$,
$$J_0 \subsetneq J_1 \subsetneq \cdots \subsetneq J_s = Q'_1$$ such that $$\frac{\mathfrak{q}_iR[[y_1, \ldots, y_n]]}{IR[[y_1, \ldots, y_n]]} = B' \cap J_0 \subsetneq B' \cap J_1 \subsetneq \cdots \subsetneq B' \cap J_s = \frac{(y_1, \ldots ,y_n)R[[y_1, \ldots, y_n]]}{IR[[y_1, \ldots, y_n]]}$$ is a chain of prime ideals of $B'$. Suppose that this chain of prime ideals of $B'$ is not saturated. Then there are prime ideals $P_0, P_1, \ldots ,P_{s + 1}$ of $B'$ satisfying
$$\frac{\mathfrak{q}_iR[[y_1, \ldots, y_n]]}{IR[[y_1, \ldots, y_n]]} = P_0 \subsetneq P_1 \subsetneq \cdots \subsetneq P_{s + 1} = \frac{(y_1, \ldots ,y_n)R[[y_1, \ldots, y_n]]}{IR[[y_1, \ldots, y_n]]}.$$ By Proposition \ref{proposition: folklore going down}, and going down property holds between $T'$ and $B'$. Therefore, there are prime ideals $P'_0, P'_1, \ldots ,P'_{s + 1}$ of $T'$ satisfying
$$P'_0 \subsetneq P'_1 \subsetneq \cdots P'_s \subsetneq P'_{s + 1} = Q'_1$$
with $B' \cap P'_{\ell} = P_{\ell}$ for all $\ell \in \{0,1, \ldots , s + 1\}$. Now, $P'_0$ contains a minimal prime ideal $P''$ of $T'$, and we have $$B' \cap P'' \subseteq B' \cap P'_0 = P_0 = \frac{\mathfrak{q}_iR[[y_1, \ldots, y_n]]}{IR[[y_1, \ldots, y_n]]}$$ and so
$$B' \cap P'' = \frac{\mathfrak{q}_iR[[y_1, \ldots, y_n]]}{IR[[y_1, \ldots, y_n]]}.$$
Because of the correspondence between the minimal prime ideals of $T'$ contained in $Q'_1$ and the minimal prime ideals of $B'$ contained in $B' \cap Q'_1$, we have that $P'' = J_0.$ It follows that there is a saturated chain of prime ideals of $T'$ starting at $J_0$ and ending at $Q'_1$ of length longer than $s$. This violates that $T'_{Q'_1}$ is catenary. Thus the chain
$$\frac{\mathfrak{q}_iR[[y_1, \ldots, y_n]]}{IR[[y_1, \ldots, y_n]]} = B' \cap J_0 \subsetneq B' \cap J_1 \subsetneq \cdots \subsetneq B' \cap J_s = B' \cap Q'_1$$
is saturated.

Because $T'_{Q'_1}$ is catenary, the coheight of $J_0$ in $T'_{Q'_1}$ is $s$. Similarly, since $B'_{B' \cap Q'_1}$ is catenary, the coheight of $\mathfrak{q}_iR[[y_1, \ldots, y_n]]/IR[[y_1, \ldots, y_n]]$ in $B'$ is $s$. It follows that there is a coheight preserving bijection between $\Min T'_{Q'_1}$ and $\Min B'_{B' \cap Q'_1}$. A similar argument shows that there is a coheight preserving bijection between $\Min T'_{Q'_2}$ and $\Min B'_{B' \cap Q'_2} = \Min B'_{B' \cap Q'_1}$.

Now suppose $T$ is quasi-excellent. Since a quotient ring of a quasi-excellent ring is also quasi-excellent, $\overline{T}$ is quasi-excellent. By Proposition \ref{quasiexcellent}, $R$ is quasi-excellent. Noting that, if $A$ is a quasi-excellent local ring then so is $A[[y]]$, it follows that $B'$ is quasi-excellent.

Finally, observe that $B = \phi^{-1}(B')$ satisfies the conclusions of the theorem.
\end{proof}

The conclusions of Theorem \ref{theorem: generalized gluing alt} allow us to apply it multiple times. We illustrate with an example.

\begin{example}\label{secondexample}
        Let
        $T = \C[[x_1, x_2, x_3, x_4, y_1,y_2,y_3,y_4,z_1,z_2,z_3,z_4]]/J$ where $$J = (x_1, y_1, z_1) \cap (x_1, y_1,z_2,z_3) \cap (x_1, y_2,y_3, z_1) \cap (x_1, y_2,y_3,z_2,z_3) \, \cap $$ 
        \vspace{-1.25cm}
        $$(x_2,x_3, y_1, z_1) \cap (x_2,x_3, y_1,z_2,z_3) \cap (x_2,x_3, y_2,y_3, z_1) \cap (x_2,x_3, y_2,y_3,z_2,z_3) $$
        and let $Q_1 = (x_1, y_1,y_2,y_3,y_4,z_1,z_2,z_3,z_4)/J$ and $Q_2 = (x_2, x_3,y_1,y_2,y_3,y_4,z_1,z_2,z_3,z_4)/J$. 
Then $$\overline{T} = \frac{T}{Q_1 \cap Q_2} \cong \frac{\C[[x_1, x_2,x_3, x_4]]}{(x_1) \cap (x_2, x_3)}.$$
Now,$$\frac{\overline{T}[[y_1, y_2,y_3, y_4,z_1,z_2,z_3,z_4]]}{(y_1,z_1) \cap (y_1,z_2,z_3) \cap (y_2,y_3,z_1) \cap (y_2,y_3,z_2,z_3)} \cong T,$$
and the conditions for Theorem \ref{theorem: generalized gluing alt} are met. Hence there is a local subring $T_1$ of $T$ satisfying the conclusions of Theorem \ref{theorem: generalized gluing alt}. In particular, $$T_1 \cong \frac{R[[y_1, y_2,y_3, y_4,z_1,z_2,z_3,z_4]]}{(y_1,z_1) \cap (y_1,z_2,z_3) \cap (y_2,y_3,z_1) \cap (y_2,y_3,z_2,z_3)}$$ for some domain $R$ containing the rationals, and $T_1$ satisfies the hypotheses of the theorem. Moreover, since $T$ is excellent, $T_1$ is quasi-excellent. Apply Theorem \ref{theorem: generalized gluing alt} again with $$Q_1 = \frac{(y_1, z_1,z_2,z_3,z_4)}{(y_1,z_1) \cap (y_1,z_2,z_3) \cap (y_2,y_3,z_1) \cap  (y_2,y_3,z_2,z_3)}$$ and
$$Q_2 = \frac{(y_2,y_3, z_1,z_2,z_3,z_4)}{(y_1,z_1) \cap (y_1,z_2,z_3) \cap (y_2,y_3,z_1) \cap (y_2,y_3,z_2,z_3)}$$ to get a quasi-excellent local subring $B$ of $T_1$ satisfying the conclusions of Theorem \ref{theorem: generalized gluing alt}. Note that verifying the details for this application of Theorem \ref{theorem: generalized gluing alt} is very similar to the details given in Example \ref{mainexample}. Figure 3 shows a partially ordered set that is embedded in the prime spectrum of the ring $B$. We can now apply Theorem \ref{theorem: GluingTheorem} to the ring $B$ to get a quasi-excellent local ring $B_1$ where the prime spectrum of $B_1$ is the same as the prime spectrum of $B$ except that, for the prime spectrum of $B_1$, the two minimal prime ideals of $B$ are glued together. So, for this complete local ring $T$, we have answered Question \ref{question: easier classication question} from the Introduction for two different partially ordered sets; the one pictured in Figure 3 and the one obtained from Figure 3 by gluing the two minimal elements.
\end{example}

\begin{center}

\begin{figure}

\begin{tikzpicture}

%NODES
\node[main] (1U) {};
\node[main] (4U) [below right of=1U] {};
\node[main] (5U) [below of=4U] {};
\node[main] (6U) [below left of= 5U] {};
\node[main] (2U) [below left of=4U, left of=4U] {};

%\node[main] (1) [below of=6U] {};
\node[main] (4) [below right of=6U] {};
\node[main] (5) [below of=4] {};
\node[main] (6) [below left of= 5] {};
\node[main] (2) [below left of=4, left of=4] {};

\node[main] (8) [below right of=6] {};
\node[main] (9) [below of=8] {};
%\node[main] (10) [below left of= 9] {10};
\node[main] (7) [below left of=8, left of=8] {};
\node[main] (10L) [below of = 7] {};
\node[main] (10R) [below of =9] {};

%DRAWING
\draw[line width=0.25mm] (1U) -- (2U) -- (6U);
\draw[line width=0.25mm] (1U) -- (4U) -- (5U) -- (6U);
\draw[line width=0.25mm] (6U) -- (4) -- (5) -- (6) -- (8) -- (9) -- (10R);
\draw[line width=0.25mm] (6U) -- (2) -- (6) -- (7) -- (10L);
\end{tikzpicture}
\caption{Partial $\Spec(B)$}
\label{DrawingOfX}
\end{figure}

\end{center}

\end{document}